\documentclass{amsart}
\usepackage{graphicx}
\usepackage{enumerate}
\usepackage{fullpage}
\usepackage{subfigure}
\newtheorem{theorem}{Theorem}[section]

\newtheorem{proposition}[theorem]{Proposition}
\theoremstyle{definition}

\theoremstyle{remark}

\numberwithin{equation}{section}
\newcommand{\set}[1]{\left\{#1\right\}}
\newcommand{\comb}[2]{\left( \begin{array}{c} #1 \\ #2 \end{array} \right)}
\newcommand{\R}{\mathbb R}
\newcommand{\C}{\mathbb C}

\newcommand{\HH}{{\mathcal H}}

\newcommand{\hyper}[2]{{}_{_{#1}}F_{_{#2}}}
\begin{document}

\title[]{On the spectrum of the transfer operators of a one-parameter family with intermittency transition}%
\author{S. Ben Ammou} 
\address{Faculty of Science, Computational Mathematics Laboratory, University of Monastir, Monastir 5000, Tunisia}
\author{C. Bonanno} 
\address{Dipartimento di Matematica, Universit\`a di Pisa, Pisa, Italy}
\email{bonanno@dm.unipi.it}
\author{I. Chouari}
\address{Faculty of Science, Computational Mathematics Laboratory, University of Monastir, Monastir 5000, Tunisia}
\author{S. Isola}
\address{Dipartimento di Matematica e Informatica, Universit\`a di Camerino, Camerino (MC), Italy}
\thanks{The second author is partially supported by ``Gruppo Nazionale per l'Analisi Matematica, la Probabilit\`a e le loro Applicazioni (GNAMPA)'' of Istituto Nazionale di Alta Matematica (INdAM), Italy.}
\maketitle

% ----------------------------------------------------------------
\begin{abstract}
We study the transfer operators for a family $F_r:[0,1] \to [0,1]$ depending on the parameter $r\in [0,1]$, which interpolates between the tent map and the Farey map. In particular, considering the action of the transfer operator on a suitable Hilbert space, we can define a family of infinite matrices associated to the operators and study their spectrum by numerical methods.
\end{abstract}

% ----------------------------------------------------------------

\section{Introduction} 
Let $F:[0,1]\to [0,1]$ be the \emph{Farey map} defined by
\begin{equation} \label{farey}
F(x)=\left\{
\begin{array}{ll}
\frac{x}{1-x} & \mbox{if }\ 0\le x\le \frac{1}{2}\\[0.3cm]
\frac{1-x}{x} & \mbox{if }\ \frac{1}{2} \le x \le 1
\end{array} \right.
\end{equation}
From the ergodic point of view it is of interest since it is expanding everywhere but at the fixed point $x=0$ where it has slope one. This make this map a simple model of the physical phenomenon of \emph{intermittency} \cite{PM}. Moreover, the Farey map is related to the \emph{Farey fractions} and encodes the continued fraction algorithm (see \cite{BI1}).

An effective tool in the study of the dynamics induced by a map on the interval is provided by the \emph{transfer operator} associated to the map (see \cite{Ba} for an overview), a linear operator whose spectrum on a suitable space of functions gives information about the existence of invariant measures, ergodicity and mixing. For expanding maps it turns out that the transfer operator is quasi-compact when restricted to a space of sufficiently regular functions, hence the spectrum is made of isolated eigenvalues with finite multiplicity and the essential part, a disc of radius strictly smaller than the spectral radius. Instead for intermittent maps as the Farey map, even when restricted to $C^\infty$ functions the essential spectrum of the transfer operator is equal to the whole spectrum. In particular the ergodic properties of the map cannot be deduced by this approach. 

The spectrum of a family of transfer operators of the Farey map has been studied in \cite{pap1,BGI,Is} acting on a suitable Hilbert space $\HH$ of holomorphic functions. The operators studied in these papers are self-adjoint and positive, and it turns out that the spectrum consists of the interval $[0,1]$ plus an isolated real eigenvalue.

Another approach to the properties of the transfer operator of the Farey map has been introduced in \cite{GI}. It has been first noticed that the \emph{Minkowski question mark function} conjugates $F$ with the \emph{tent map} $T:[0,1]\to [0,1]$ defined by
\begin{equation} \label{tent}
T(x)=\left\{
\begin{array}{ll}
2x & \mbox{if }\ 0\le x\le \frac{1}{2}\\[0.3cm]
2(1-x)& \mbox{if }\ \frac{1}{2} \le x \le 1
\end{array} \right.
\end{equation}
Then the authors have introduced a one-parameter family $F_r:[0,1]\to [0,1]$ of expanding maps with $r\in [0,1]$, interpolating between $T$ and $F$, namely $F_0=T$ and $F_1=F$. The family $F_r$ is defined by
\begin{equation} \label{farey-r}
F_{r}(x)=\left\{
\begin{array}{ll}
\frac{(2-r)x}{1-rx} & \mbox{if }\ 0\le x\le \frac{1}{2}\\[0.3cm]
\frac{(2-r)(1-x)}{1-r+rx} & \mbox{if }\ \frac{1}{2} \le x \le 1
\end{array} \right.
\end{equation}
A thermodynamic approach to the properties maps $F_r$ has been considered in \cite{DIK}, and we also refer to \cite{kess} for a recent study of this family.

In \cite{GI}, the authors have studied the transfer operators $\mathcal{P}_{r}$ for the family $F_r$ acting on the Hilbert space $\HH$. It turns out that $\mathcal{P}_{r}$ is of the trace-class for all $r\in [0,1)$ and the trace can be analytically computed. Moreover they have discussed the possibility that the spectral properties of $\mathcal{P}_{1}$, namely the transfer operator for the Farey map, are approximated by those of $\mathcal{P}_{r}$ letting $r\to 1^-$.

In this paper we use the matrix approach that has been introduced in \cite{pap1} to study the operators $\mathcal{P}_{r}$,  and discuss possible insights about the spectral properties of $\mathcal{P}_{1}$.

\section{Transfer operators} 

The transfer operator $\mathcal{P}_r$ associated to the map $F_{r}$ acts on functions $f:[0,1]\to \C$ as
\[
(\mathcal{P}_r f)(x) := \sum_{y\, :\, F_r(y)=x}\, \frac{f(y)}{|F'_r(y)|}
\]
which using \eqref{farey-r} becomes
\begin{equation} \label{pr}
(\mathcal{P}_r f)(x)=(\mathcal{P}_{r,0}f+\mathcal{P}_{r,1}f)(x)
\end{equation}
with
\begin{equation}  \label{p0p1}
(\mathcal{P}_{r,0}f)(x)=\frac{\rho}{(\rho+rx)^{2}}\, f\Big(\frac{x}{\rho+rx}\Big)\quad \text{and} \quad (\mathcal{P}_{r,1}f)(x)=\frac{\rho}{(\rho+rx)^{2}}\, f\Big(1-\frac{x}{\rho+rx}\Big)
\end{equation} 
where $\rho:= 2-r$, a notation that will be used in the rest of the paper.

The operator $\mathcal{P}_{1}$ for the Farey map $F$ has been studied in \cite{pap1,BGI} on the Hilbert space $\HH$ of holomorphic functions defined as
\begin{equation}\label{spazio}
\HH:= \set{f:[0,1]\to \C\, :\, f=\mathcal{B}[\varphi]\ \text{for some } \varphi\in L^{2}(m)}
\end{equation}
where $\mathcal{B}[\cdot]$ denotes the generalized Borel transform
\begin{equation} \label{trans}
(\mathcal{B}[\varphi])(x):= \frac{1}{x^{2}}\int_{0}^{\infty}e^{-\frac{t}{x}}\, e^{t}\, \varphi(t)\, dm(t)\, ,
\end{equation}
and $L^{2}(m):= L^2(\R^+,m)$ where $m$ is the measure on $\R^+$
\[
dm(t)=te^{-t}dt.
\]
The space $\HH$ is endowed with the inner product inherited by the inner product on $L^2(m)$ through the $\mathcal{B}$-transform, that is
\begin{equation}\label{inn}
(f_{1},f_{2})_{\HH}:=\int_{0}^{\infty}\varphi_{1}(t)\, \overline{\varphi_{2}(t)}\,dm(t)\qquad \text{if}\quad f_{i}=\mathcal{B}[\varphi_{i}]\, .
\end{equation}
In \cite{GI}, the authors have studied the operators $\mathcal{P}_{r}$ on $\HH$ for all $r\in [0,1)$ and have proved many properties that we collect in the following theorem.

\begin{theorem}[\cite{GI}] \label{teo-gi}
For all $r\in [0,1)$ the space $\HH$ is invariant for $\mathcal{P}_{r}$ and
\[
\mathcal{P}_r \, \mathcal{B}[\varphi]=\mathcal{B}[(M_{r}+N_{r})\varphi]
\]
for all $\varphi \in L^2(\R^+,m)$, where $M_{r},N_r :L^2(m) \rightarrow L^2(m)$ are defined as
\begin{equation}\label{trad}
(M_{r}\varphi)(t)=\frac{1}{\rho}\, e^{-\frac{r}{\rho}t}\, \varphi\Big(\frac{t}{\rho}\Big) \qquad \text{and} \qquad (N_{r}\varphi)(t)=\frac{1}{\rho}\, e^{\frac{1-r}{\rho}t}\, \int_{0}^{\infty}J_{1}\left(2\sqrt{st/\rho}\right)\, \sqrt{\frac{\rho}{st}}\ \varphi(s)\, dm(s)
\end{equation}
where $J_{q}$ denotes the Bessel function of order $q$. Moreover
\begin{enumerate}[(i)]
\item the function
\[
g_r(x) = \frac{1}{1-r+rx} = \mathcal{B}\left[ \frac{1-e^{-\frac{r}{1-r}t}}{rt} \right](x)\ \in \HH
\]
satisfies $\mathcal{P}_r g_r = g_r$, hence it is the density of an absolutely continuous invariant measure for $F_r$;
\item the operators $M_r$ and $N_r$ on $\HH$ are of trace-class, hence the same holds for $\mathcal{P}_{r}$;
\item the spectra of $M_r$ and $N_r$ contain only simple eigenvalues, in particular
\[
\text{sp}(M_r) = \{0\} \cup \left\{ \rho^{-k} \right\}_{k\ge 1}\quad \text{and} \quad \text{sp}(N_r) =\{0\} \cup \left\{ (-1)^{k-1}\left( \frac{4\rho}{(1+\sqrt{1+4\rho})^{2}}\right)^{k} \right\}_{k\ge 1}\, ;
\]
\item the trace of $\mathcal{P}_{r}$ can be explicitly computed and is given by
\[
\text{trace}\,(\mathcal{P}_r) =\frac{1}{1-r}+\frac{\sqrt{1+4\rho}-1}{2\sqrt{1+4\rho}}\, .
\]
\end{enumerate}
\end{theorem}

\section{The matrix approach} 

As shown in \cite{BGI}, the Hilbert space $L^2(m)$ admits a complete orthogonal system $\{e_n\}_{n\ge 0}$ given by the Laguerre polynomials defined as
\begin{equation} \label{laguerre}
e_n(t) := \sum_{m=0}^n\, \comb{n+1}{n-m}\, \frac{(-t)^m}{m!}
\end{equation}
which satisfy 
\[
(e_{n},e_{n})=\frac{\Gamma(n+2)}{n!} = n+1
\]
for all $n\ge 0$. Hence, using Theorem \ref{teo-gi}, we can study the action of $\mathcal{P}_r$ on $\HH$ by the action on $L^2(m)$ of an infinite matrix representing the operators $P_r:=M_r+N_r$ defined in \eqref{trad} for the basis $\{e_n\}_{n\ge 0}$. That is for any $\phi\in L^{2}(m)$, we can write
\[
\phi(t)=\sum_{n=0}^{\infty}\phi_{n}e_{n}(t) \quad \text{with} \quad \phi_n = \frac{1}{n+1}\, (\phi,e_{n})
\] 
hence $\phi$ is an eigenfunction of $P_{r}$ with eigenvalue $\lambda$ if and only if
\[
(P_{r}\phi,e_{k})=\lambda\, (\phi,e_{k})=\lambda\, (k+1)\, \phi_{k} \qquad \forall \; k\geq 0
\]
Using the notation $c_{kn}^{r}:= (P_{r}e_{n},e_{k})$ we obtain that 
\begin{equation} \label{eig-mat}
P_{r}\phi=\lambda \phi \quad \Leftrightarrow\quad C_{r}\phi=\lambda D\phi \quad \Leftrightarrow\quad A_{r}\phi=\lambda \phi
\end{equation}
where $C_{r}$ and $D$ are given by 
\[
C_{r}=(c_{kn}^{r})_{k,n\geq 0}\quad \text{and} \quad D=\text{diag}(k+1)_{k\ge 0}
\]
and $A_{r}$ is the infinite matrix
\begin{equation}\label{ar}
A_{r}= (a_{kn}^{r})_{k,n\geq 0}\qquad \text{with} \quad a_{kn}^r = \frac{c_{kn}^{r}}{k+1}\, .
\end{equation}
We now use the definitions \eqref{trad} of the operators $M_r$ and $N_r$ to compute
\begin{proposition}\label{calcoli}
For all $r\in (0,1)$ we have
\[
\begin{aligned}
a_{kn}^r = \ & \comb{n+k+1}{n}\, \frac{(2-r)r^{k}}{2^{n+k+2}}\, \hyper{2}{1}\left(-k,-n;-k-n-1,\frac{2(r-1)}{r}\right)+\\[0.2cm]
&+\sum_{l=0}^{n}(-1)^{l}\, \comb{n+1}{n-l}\, \comb{l+k+1}{l}\, \frac{(2-r)r^{k}}{2^{l+k+2}}\, \hyper{2}{1}\left(-k,-l;-k-l-1,\frac{2(r-1)}{r}\right)
\end{aligned}
\]
for all $k,n\ge 0$, where $\hyper{2}{1}$ denotes the hypergeometric function.
\end{proposition}

\begin{proof}
From the definition of the $a_{kn}^r$ in \eqref{ar}, we first have to compute the terms $c_{kn}^r = (P_{r}e_{n},e_{k})$. We have
\[
c_{kn}^{r}=((M_{r}+N_{r})e_{n},e_{k}) =\int_{0}^{+\infty}(M_{r}e_{n})(t)e_{k}(t)dm(t)+\int_{0}^{+\infty}(N_{r}e_{n})(t)e_{k}(t)dm(t)
\]
and compute the two integrals separately. For the first we find
\[
\begin{aligned}
& \int_{0}^{+\infty}(M_{r}e_{n})(t)e_{k}(t)dm(t) = \int_{0}^{+\infty}\frac{1}{\rho}\, e^{-\frac{r}{\rho}t}\, e_{n}\left(\frac{t}{\rho}\right)\, e_{k}(t)\, te^{-t}dt=\frac{1}{\rho}\, \int_{0}^{+\infty}e^{-t(\frac{r}{\rho}+1)}\, t\, e_{n}\left(\frac{t}{\rho}\right)\, e_{k}(t)dt =\\[0.2cm]
&=\frac{1}{\rho}\frac{\Gamma(n+k+2)\, (\frac{r}{\rho}+1-\frac{1}{\rho})^{n}\, (\frac{r}{\rho}+1-1)^{k}}{k!\, n!\, (\frac{r}{\rho}+1)^{n+k+2}}\hyper{2}{1}\left(-k,-n;-k-n-1,\frac{(\frac{r}{\rho}+1)(\frac{r}{\rho}+1-\frac{1}{\rho}-1)}{(\frac{r}{\rho}+1-\frac{1}{\rho})(\frac{r}{\rho}+1-1)}\right) =\\[0.2cm]
&= \frac{\Gamma(n+k+2)}{k!\, n!}\, \frac{(2-r)r^{k}}{2^{n+k+2}}\, \hyper{2}{1}\left(-k,-n;-k-n-1,\frac{2(r-1)}{r}\right)
\end{aligned}
\]
where in the second line we have used \cite[equation 7.414 (4), p. 809]{GR}.

For the second integral we use \eqref{trad} and the polynomial expression \eqref{laguerre} for $e_n$ to write
\[
\begin{aligned}
& \int_{0}^{+\infty}(N_{r}e_{n})(t)\, e_{k}(t)\, dm(t) =\int_{0}^{+\infty}\frac{1}{\rho}\, e^{\frac{1-r}{\rho}t}\biggl[\int_{0}^{+\infty}\frac{J_{1}(2\sqrt{st/\rho})}{\sqrt{st/\rho}}\, e_{n}(s)\, dm(s) \biggr] e_{k}(t)\, te^{-t}dt= \\[0.2cm]
&=\int_{0}^{+\infty}\frac{1}{\rho}\, e^{\frac{1-r}{\rho}t}\biggl[\int_{0}^{+\infty}\frac{J_{1}(2\sqrt{st/\rho})}{\sqrt{st/\rho}}\sum_{l=0}^{n}\, \comb{n+1}{n-l}\, \frac{(-s)^{l}}{l!}\, dm(s) \biggr]e_{k}(t)\, te^{-t}dt= \\[0.2cm]
&=\int_{0}^{+\infty}\frac{1}{\rho}\, e^{\frac{1-r}{\rho}t}\biggl[\sum_{l=0}^{n}\comb{n+1}{n-l}\, \frac{(-1)^{l}}{l!}\left(\frac{\rho}{t}\right)^{l+2}\, \int_{0}^{+\infty}J_{1}(2\sqrt{u})\, u^{l+\frac 12}\, e^{-\frac{\rho}{t} u}\, du \biggr] e_{k}(t)\, te^{-t}dt \\[0.2cm]
\end{aligned}
\]
where in the last line we have made the change of variable $u= \frac t \rho\, s$. Using now \cite[equation 6.643 (4), p. 709]{GR} to write
\[
\int_{0}^{+\infty}J_{1}(2\sqrt{u})u^{l+1/2}e^{-\frac{\rho }{t}u}\, du = l!\, e^{-\frac t \rho}\, \left(\frac{\rho}{t}\right)^{-l-2}\, e_l\left(\frac t \rho\right)\, ,
\]
we obtain
\[
\begin{aligned}
&\int_{0}^{+\infty}(N_{r}e_{n})(t)\, e_{k}(t)\, dm(t) =\int_{0}^{+\infty}\frac{1}{\rho}\, e^{\frac{1-r}{\rho}t}\biggl[\sum_{l=0}^{n}\comb{n+1}{n-l}\, (-1)^{l}\, e^{-\frac{t}{\rho}}\, e_{l}\left(\frac{t}{\rho}\right) \biggr] e_{k}(t)\, te^{-t}dt = \\[0.2cm]
&=\sum_{l=0}^{n}\, (-1)^{l}\, \comb{n+1}{n-l}\, \frac{1}{\rho}\, \int_{0}^{+\infty}t\, e^{-\frac{2}{\rho}t}\, e_{l}\left(\frac{t}{\rho}\right)\,  e_{k}(t)\, dt =\\[0.2cm] 
&=\sum_{l=0}^{n}\, (-1)^{l}\, \comb{n+1}{n-l}\, \frac{1}{\rho}\, \frac{\Gamma(l+k+2)}{l!\, k!}\frac{(\frac{1}{\rho})^{l}\, (\frac{2}{\rho}-1)^{k}}{(\frac{2}{\rho})^{l+k+2}} \, \hyper{2}{1}\left(-k,-l;-k-l-1,\frac{\frac{2}{\rho}(\frac{1}{\rho}-1)}{\frac{1}{\rho}\, (\frac{2}{\rho}-1)}\right)= \\[0.2cm]
&=\sum_{l=0}^{n}\, (-1)^{l}\, \comb{n+1}{n-l}\, \frac{\Gamma(l+k+2)}{l!\, k!}\frac{(2-r)r^{k}}{2^{l+k+2}}\, \hyper{2}{1}\left(-k,-l;-k-l-1,\frac{2(r-1)}{r}\right)
\end{aligned}
\]
where in the third line we have used again \cite[equation 7.414 (4), p. 809]{GR}.

The proof is finished by adding the two integrals and dividing by $k+1$.
\end{proof}

We can now numerically approximate solutions to \eqref{eig-mat} using the standard north-west corner approximation of the matrices $A_r$. For $N\ge 1$ we let $A_{r,N}$ denote the $N\times N$ matrix defined as
\[
A_{r,N}= (a_{kn}^{r,N})_{k,n=0,\dots,N-1}\qquad \text{with} \quad a_{kn}^{r,N} = a_{kn}^{r}\, .
\]
For each $r\in (0,1)$ we find $N$ eigenvalues for $A_{r,N}$ which approximate the eigenvalues of $P_r = M_r + N_r$, whose spectrum consists only of the point spectrum as stated in Theorem \ref{teo-gi}. In Figure \ref{fig-eig} we have plotted the eigenvalues of $A_{r,N}$ for $N=50$ as functions of $r\in (0,1)$. 

\begin{figure}[h]
\begin{center}
\includegraphics[width=10.0cm,keepaspectratio]{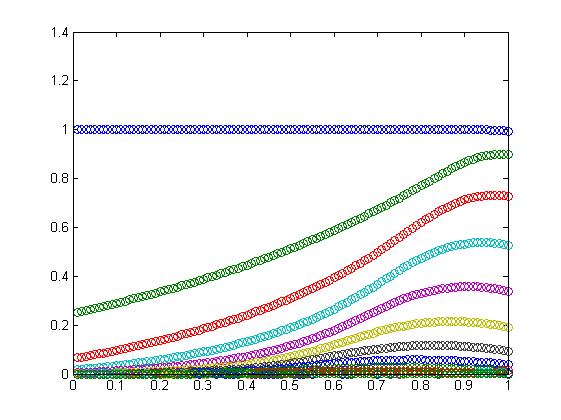}
\caption{The eigenvalues of $A_{r,N}$ for $N=50$ as functions of $r\in (0,1)$.}
\label{fig-eig}
\end{center}
\end{figure}

First of all we notice that for all $r$ we find $\lambda=1$ as leading eigenvalue. This is expected since by Theorem \ref{teo-gi}(i) for all $r$ the operator $\mathcal{P}_r$ admits an eigenfunction $g_r$ with eigenvalue $1$, which corresponds to the invariant measure for the map $F_r$. Second we conclude from the results that the approximation of the eigenvalues of $P_r$ gets worse and worse as $r$ approaches 1. This is evident in the leading eigenvalue, which is a curve very close to 1, but slightly decreasing as $r\to 1^-$, and is probably the reason for the other curves to be non-increasing as $r$ approaches 1.

Then we can ask how good is the approximation given by the eigenvalues in Figure \ref{fig-eig}. To quantify the goodness of the approximation we can use the only analytical result about the spectrum of $P_r$, namely the computation of its trace given in Theorem \ref{teo-gi}(iv). We have plotted in Figure \ref{convergence}(a) the function given in Theorem \ref{teo-gi}(iv), which is the upper most curve diverging as $r\to 1^-$, and the traces of the matrices $A_{r,N}$ as functions of $r\in (0,1)$ for different values of $N$. In particular we have chosen $N=10,20,30,40,50,60$. We see that as $N$ increases we get a better and better approximation of the trace of $P_r$, and the two almost coincide for $r<0.6$. However as $r$ approaches 1, we see that the approximation of the trace becomes poor, and in particular for $N=60$, the small oscillations in the curve show numerical instabilities in the computations.
   
Hence we believe that the computation of the eigenvalues of $A_{r,N}$ in Figure \ref{fig-eig} is a very good approximation of the eigenvalues of $P_r$ for $r<0.6$. For example for $r=0$, the eigenvalues of $P_0$ are the set $\set{2^{-2k}}_{k\ge 0} \cup \{0\}$ (see e.g. \cite[Proposition 4.7]{GI}), and these values coincide for what we find in Figure \ref{fig-eig}. It is instead unclear what happens as $r$ approaches 1.

To try to understand this point we have plotted in Figure \ref{convergence}(b) the eigenvalues of Figure \ref{fig-eig} together with the sums of the eigenvalues of $M_r$ and $N_r$ for $k=1,3,5,7,9$ find in Theorem \ref{teo-gi}(iii), that is the curves
\[
\rho^{-k} + (-1)^{k-1}\left( \frac{4\rho}{(1+\sqrt{1+4\rho})^{2}}\right)^{k}\, , \qquad k=1,3,5,7,9\, ,
\]
which coincide with the first five eigenvalues of $P_r$ for $r=0$. The behavior of the trace of $A_{r,N}$ with respect to that of $P_r$, and Figure \ref{convergence}(b) suggest that all the eigenvalues of $P_r$ converge to 1 as $r\to 1^-$, so that for $r=1$ the spectrum of $P_r$ would consist of the eigenvalues 0 and 1 and of the purely continuous spectrum $(0,1)$. Hence our numerical results support the conjecture given in \cite{Is}.
 
\begin{figure}[h]
    \begin{center}
    \subfigure[]
    {\includegraphics[width=6.0cm,keepaspectratio]{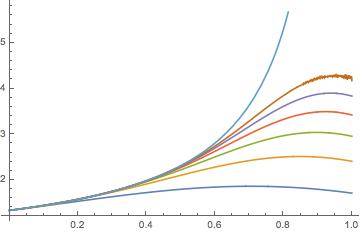}}
    \hspace{0.7cm}
    \subfigure[]
    {\includegraphics[width=6.0cm,keepaspectratio]{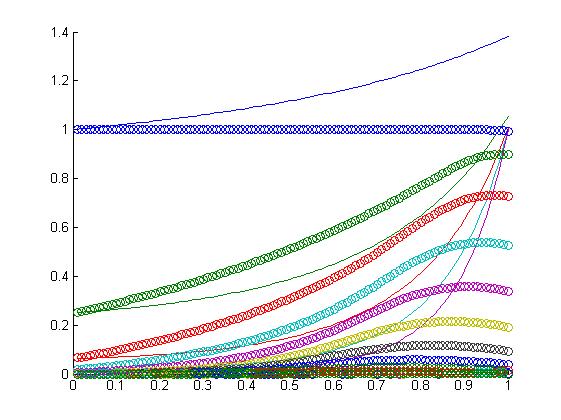}}
    \caption{(a) The trace of the operators $P_r$ (upper most curve) and in increasing order, the trace of the matrix $A_{r,N}$ for $N=10,20,30,40,50,60$, as functions of $r\in (0,1)$. (b) The eigenvalues of $A_{r,N}$ for $N=50$ compared with the sums of the eigenvalues of $M_r$ and $N_r$ for $k=1,3,5,7,9$, as  functions of $r\in (0,1)$.}
    \label{convergence}
    \end{center}
\end{figure}

% ----------------------------------------------------------------


\begin{thebibliography}{ABC}

\bibitem{Ba} V. Baladi, ``Positive Transfer Operators and Decay of correlations'', World Scientific, 2000

\bibitem{pap1} S. Ben Ammou, C. Bonanno, I. Chouari, S. Isola, \emph{On the leading eigenvalue of transfer operators of the Farey map with real temperature}, Chaos Solitons Fractals \textbf{71} (2015), 60--65

\bibitem{BGI} C. Bonanno, S. Graffi, S. Isola, \emph{Spectral analysis of transfer operators associated to Farey fractions}, Atti Accad. Naz. Lincei Cl. Sci. Fis. Mat. Natur. Rend. Lincei (9) Mat. Appl. \textbf{19} (2008), 1--23

\bibitem{BI1} C. Bonanno, S. Isola, \emph{Orderings of rationals and dynamical systems}, Colloq. Math. \textbf{116} (2009), 165--189

\bibitem{DIK} M. Degli Esposti, S. Isola, A. Knauf, \emph{Generalized Farey trees, transfer operators and phase transitions}, Commun. Math. Phys. \textbf{275} (2007), 297--329

\bibitem{GI} M. Giampieri, S. Isola, \emph{A one parameter family of analytic Markov maps with an intermittency transition}, Discrete Cont. Dyn. Syst. \textbf{12} (2005), 115--136

\bibitem{GR} I. Gradshteyn, I. Ryzhik, ``Table of integrals, series and products'', Academic Press, 1965.

\bibitem{Is} S. Isola,  \emph{On the spectrum of Farey and Gauss maps}, Nonlinearity \textbf{15} (2002), 1521--1539

\bibitem{kess} J. Kautzsch, M. Kesseb\"ohmer, T. Samuel, \emph{On the convergence to equilibrium of unbounded observables under a family of intermittent interval maps}, arXiv:1410.3805 [math.DS]

\bibitem{PM} Y. Pomeau, P. Manneville, \emph{Intermittency transition to turbulence in dissipative dynamical systems}, Commun. Math. Phys. \textbf{74} (1980), 189--197

\end{thebibliography}
\end{document}